\documentclass[11pt, a4paper]{amsart}

\usepackage[pdftex]{hyperref}

\hypersetup{ colorlinks=true, linkcolor=blue, filecolor=magenta, urlcolor=cyan, }

\hypersetup{
     colorlinks=true,       
    linkcolor=blue,          
    citecolor=cyan,        
    filecolor=cyan,         
    urlcolor=magenta        
}

\newcommand{\RR}{\mathbb{R}} 
\newcommand{\CC}{\mathbb{C}} 

\newcommand{\Aut}{\mathrm{Aut}}

\newcommand{\inp}{\,\raisebox{0.4ex}{$\lrcorner$}\;}
\newcommand{\UG}{\mathrm{U}}
\newcommand{\mh}{\mathcal H}

\newcommand{\mrt}{P}

\newcommand{\W}{\mathcal W}
\newcommand{\V}{\mathcal V}
\newcommand{\cU}{\mathcal U}
\newcommand{\vp}{\varphi}

\newcommand{\C}{\mathbb{C}}

\newcommand{\paren}[1]{\left(#1\right)}
\newcommand{\abs}[1]{\left\lvert#1\right\rvert}

\newcommand{\set}[1]{\left\{#1\right\}}

\newcommand{\pd}[2]{\frac{\partial#1}{\partial#2}}

\newcommand{\pdl}[2]{\partial#1/\partial#2}

\newcommand{\ind}[4]{{#1}^{\phantom{#2}#3}_{#2\phantom{#3}#4}}

\newcommand{\wt}{\widetilde}

\newcommand{\w}{\wedge}
\newcommand{\beq}{\begin{eqnarray*}}
\newcommand{\eeq}{\end{eqnarray*}}
\newcommand{\beg}{\begin{equation}}
\newcommand{\eeg}{\end{equation}}

\newcommand{\di}{\partial}
\newcommand{\la}{\lambda}

\newcommand{\al}{\alpha}
\newcommand{\be}{\beta}
\newcommand{\ga}{\gamma}

\newtheorem{thm}{Theorem}[section]

\newtheorem{lem}[thm]{Lemma}
\newtheorem{prop}[thm]{Proposition}
\theoremstyle{definition}

\theoremstyle{remark}
\newtheorem{remark}[thm]{Remark}

\numberwithin{equation}{section}

\title{Almost CR manifolds with contracting CR automorphism}

\author{Jae-Cheon Joo}
\address{Department of Mathematics, King Fahd University of Petroleum and Minerals, 31261 Dhahran, Saudi Arabia}
\email{jcjoo@kfupm.edu.sa}

\author{Kang-Hyurk Lee}
\address{Department of Mathematics and Research Institute of Natural Science, Gyeongsang National University, Jinju, Gyeongnam, 52828, Republic of Korea}
\email{nyawoo@gnu.ac.kr}

\thanks{Research of the second named author was supported by  the National Research Foundation of Korea (NRF) grant funded by the Korea government (No. NRF-2019R1F1A1060891).}

\begin{document}

\begin{abstract} 
In this paper, we deal with a strongly pseudoconvex almost CR manifold with a CR contraction. We will prove that the stable manifold of the CR contaction is CR equivalent to the Heisenberg group model.
\end{abstract}

\maketitle

\section{Introduction}

The aim of this paper is characterizing strongly pseudoconvex almost CR manifolds admitting a CR contraction, equivalently  pseudo-Hermitian manifolds with a contracting pseudoconformal automorphism.

In  geometry, one of main questions is how large  transformation group of manifolds are and how to chacterize manifolds with large transformation group. In conformal and CR geometry, there is an interesting history about this question.  Obata considered conformal structures of Riemannian geometry and he succeeded in  characterizing the sphere $S^n$ under the noncompact action of conformal transformation group on compact Riemannian manifolds (\cite{Obata1970, Obata1971}), and those results were followed by Webster in CR case (\cite{Webster1977}). In 1972, Alekseevskii claimed that a Riemannian manifold either compact or noncompact with a nonproper action of conformal transformation is conformally equivalent to either $S^n$ or $\RR^n$ (\cite{Alekseevskii1972}), however it turned out later that his proof contains an error. Alekseevskii's claim has been reproved by Schoen in both conformal and CR cases in 1995 (\cite{Schoen1995}; cf. \cite{Ferrand1996,Spiro1997}). His theorem is as follows.

\begin{thm}[Schoen~\cite{Schoen1995}]
\

\begin{enumerate}
\item If the conformal automorphism group of a Riemannian manifold $M$ acts nonproperly, then $M$ is conformally equivalent to either the unit sphere or the Euclidean space.
\item If a strongly pseudoconvex CR manifold $M$ whose CR automorphism group acts nonproperly on $M$, then $M$ is CR equivalent to either the Heisenberg group or the unit sphere in the complex Euclidean space
\end{enumerate}
\end{thm}


On the other hand, in complex geometry it has been also studied the characterization of domains in $\CC^n$ by noncompact automorphism group action. Among many interesting results, one of the most famous works was proved by Wong-Rosay (\cite{Wong1977,Rosay1979}):  \emph{a strongly pseudoconvex bounded domain $\Omega$ in $\CC^n$ with a noncompact automorphism group is biholomorphic to the unit ball $\mathbb{B}^n$ in $\CC^n$}. Note that a holomorphic automorphism group $\Aut(\Omega)$ of $\Omega$ extends to a CR transformation group on the boundary $\partial\Omega$ and the noncompactness of $\Aut(\Omega)$ implies that $\Aut(\Omega)$ acts nonproperly on $\partial\Omega$. Since the the unit sphere $S^{2n-1}$ in $\CC^n$ is the boundary of the unit ball and the Heisenberg group is the unit sphere without one point, $S^{2n-1}\setminus\{p\}$, the Wong-Rosay theorem can be regarded as the counterpart of Schoen's result on the CR manifold. 

The  Wong-Rosay theorem was generalized by Gaussier-Sukhov~\cite{GS2006} and the second author~\cite{LKH2006,LKH2008} in almost complex setting. Without the integrability of the almost complex structure, the unit ball with the standard complex structure is not a unique model. There are many strongly pseudconvex domains, called \emph{model strongly pseudconvex domains} whose automorphism groups are noncompact and almost complex structures are non-integrable.  In \cite{LKH2008}, there are a full classification of such domains and a description of their automorphism groups. 

As a CR counterpart, it is also natural to ask if we can generalize Schoen's result in almost CR case. In the previous paper \cite{JooLee2015}, we have proved that the boundary of each model strongly pseudoconvex domain  is a strongly pseudoconvex almost CR manifold which has a group structure that is isomorphic to that of standard Heisenberg group as groups. We call them generalized Heisenberg groups.  It also turned out that the CR automorphism group of a generalized Heisenberg group acts nonproperly since it always admit a contracting CR automorphism as the standard one. A generalized Heisenberg group is known to be  parametrized by a skew symmetric complex $n\times n$ matrix $P$ and we  denote it by $\mh_P$ (see Section~\ref{subsec:models} for defintion).

In \cite{JooLee2015}, we proved the following theorem which is a generalization of Schoen's theorem for low dimensional almost CR manifolds.




\begin{thm}[Theorem 1.2 in \cite{JooLee2015}]\label{t;jl}
Let $M$ be a manifold with strongly pseudoconvex almost CR structure of real dimension $5$ or $7$. Suppose that the CR automorphism group acts nonproperly on $M$. If $M$ is noncompact, then $M$ is CR equivalent to $\mh_P$ for a skew symmetric complex matrix $P$. If $M$ is compact, then the almost CR structure of $M$ is integrable and $M$ is CR equivalent to the unit sphere in the complex Euclidean space.
\end{thm}

The difficulty of the proof of this theorem in arbitrary dimension comes from the construction of the Yamabe equation which plays a crucial role in the proof of Theorem~\ref{t;jl} as well as Schoen's theorem \cite{Schoen1995}.

 In order to get a characterization of $\mh_P$ in arbitrary dimension by its automorphism action, we restrict our attention to the case of admitting a contracting automorphism. Recall that the generalized Heisenberg groups admit contracting CR automorphisms. The main result in this paper is as follows.

\begin{thm}\label{t;main}
Let $M$ be a strongly pseudoconvex almost CR manifold. Suppose $M$ admits a contracting CR automorphism. Then the stable manifold $\mathcal{W}$ of the contracting automorphism is CR equivalent to $\mh_P$ for some complex skew-symmetric matrix $P$. \end{thm}

The idea of the proof of Theorem \ref{t;main} is contructing a special contact form $\theta_0$ on $\mathcal W$ for which $(\mathcal W, \theta_0)$ is indeed equivalent to $(\mh_P, \theta_P)$ as pseudo-Hermitian manifolds, where $\theta_P$ is the canonical contact form on $\mh_P$ defined by \eqref{e;thetap}. The consturcution is achieved by a standard dynamical method in Section 3. 

\medskip

\noindent\textbf{Outline of the paper.} In section~\ref{sec:aCRm}, we will introduce basic notions for the strongly pseudconvex almost CR manifold, the pseudo-Hermitian structure and their equivalence problem as studied in \cite{JooLee2015}. Then the contacting CR automorpism and the stable manifold will be discussed in Section~\ref{sec:contractingCRa}. And also in Section~\ref{subsec:canonical contact form}, we will show that there is a smooth canonical contact form of the contracting CR automorphism whose pseudo-Hermitian structure is CR equivalent to the Heisenberg model for the proof of Theorem~\ref{t;main} (see Section~\ref{sec:proof}). We also characterize the ambient manifold as the standard sphere under further assumption that the contracting automorphism has another fixed point which is contracting for the inverse map in Theorem~\ref{t;thm2}.

\medskip

\noindent\textbf{Convention.} Throughout this paper, we assume that every
structure is $C^\infty$-smooth. The summation convention is always
assumed. Greek indices will be used to indicate coefficients of complex and real tensors, respectively.  We will take the bar on Greek indices to denote
the complex conjugation of the corresponding tensor coefficients:
$\overline{Z}_\alpha=Z_{\bar\alpha}$,
$\bar\omega^\alpha=\omega^{\bar\alpha}$,
$\ind{\overline{R}}{\beta}{\alpha}{\lambda\bar\mu}
=\ind{R}{\bar\beta}{\bar\alpha}{\bar\lambda\mu}$.


\section{Strongly pseudoconvex almost CR manifolds}\label{sec:aCRm}
In this section, we review the pseudo-Hermitian structure of the strongly pseudoconvex almost CR manifold and its equivalence problem as in \cite{JooLee2015}.

\subsection{The pseudo-Hermitian structure}

By an \emph{almost CR manifold}, we mean a real $(2n+1)$-dimensional manifold $M$ with a $2n$-dimensional subbundle $H$ of $TM$ equipped with an almost complex structure $J$, which is called the almost CR structure. Since $J^2 = -\mathrm{Id}$, the complexified bundle $\CC\otimes H$ splits into two eigensubbundles $H_{1,0}$ and $H_{0,1}$ of complex dimension $n$, corresponding to eigenvalues $i$ and $-i$ of $J$, respectively:
\begin{gather*}
\CC\times H=H_{1,0}\oplus H_{0,1}\;, \\
H_{1,0}=\set{X-iJX:X\in H}\;, \quad H_{0,1}=\set{X+iJX:X\in H} \;.
\end{gather*}
 A complex-valued vector field  is of type $(1,0)$ or $(0,1)$ if it is a section to $H_{1,0}$ or $H_{0,1}$, respectively. The CR structure $(H,J)$ is \emph{integrable} if the space of $(1,0)$-vector fields, $\Gamma(H_{1,0})$ is closed under the Lie bracket: $[\Gamma(H_{1,0}),\Gamma(H_{1,0})]\subset \Gamma(H_{1,0})$. If the Lie bracket of two $(1,0)$-vector fields is always a section to $\CC\times H=H_{1,0}\oplus H_{0,1}$, then we call the CR structure is \emph{partially integrable}.

A diffeomorphism $f : M\rightarrow M$ is called a \emph{CR automorphism} if $f_*H_{1,0} = H_{1,0}$. The group of all CR automorphisms of $M$ will be denoted by $\Aut_{\mathrm{CR}}(M)$. This is the topological group with the compact-open topology.

An almost CR manifold is said to be \emph{strongly pseudoconvex} if $H$ is a contact distribution and if 
\begin{equation*}
d\theta (X, JX) >0
\end{equation*}
for a contact form $\theta$ and for any nonzero vector $X\in H$. Then the pair $(M,\theta)$ is called a \emph{pseudo-Hermitian manifold}. Let $(Z_\alpha)=(Z_1,\ldots,Z_n)$ be a $(1,0)$-frame (a local frame of $H_{1,0}$). A $\CC^n$-valued $1$-form $(\theta^\al)=(\theta^1,\ldots,\theta^n)$ is called the \emph{admissible coframe} for $\theta$ with respect to $(Z_\alpha)$ if 
\begin{equation*}
\theta^\al(Z_\be) = \delta^\al_\be,\quad \theta^\al(Z_{\bar\be}) =0,\quad \theta^\al(T)=0
\end{equation*}
for $\al,\be = 1,...,n$, where $T$ is the \emph{characteristic vector field} of $\theta$, that is, a unique real tangent vector field $\theta(T)\equiv 1$ and $T\inp d\theta\equiv 0$. The admissibility of $(\theta^\al)$ is equivalent to that the differential of the contact form $\theta$ can be written by
\beg\label{e;str1}
d\theta =
 2ig_{\alpha\bar\beta}\,\theta^\alpha\wedge\theta^{\bar\beta}
 +p_{\alpha\beta}\,\theta^\alpha\wedge\theta^\beta
 +p_{\bar\alpha\bar\beta}\,\theta^{\bar\alpha}\wedge\theta^{\bar\beta}
 \eeg
where $(g_{\al\bar\be})$ is a positive-definite hermitian symmetric  matrix which is called the Levi form, and $(p_{\al\be})$ is a skew-symmetric matrix. Throughout this paper, $(g^{\alpha\bar\beta})$ stands for the inverse matrix of the Levi form $(g_{\alpha\bar\beta})$. If $(p_{\al\be})$ vanishes identically, then the almost CR structure is partially integrable; thus we call $(p_{\al\be})$ the \emph{non-partial-integrability} with respect to $\theta, \theta^\al$. 

\medskip

Let
\begin{equation*}
\tilde\theta = e^{2f}\theta
\end{equation*}
be another contact form, where $f$ is a real-valued smooth function on $M$. 

\begin{prop}[\cite{JooLee2015}]\label{p;conformal change} Let $(\theta^\al)$ be a local admissible coframe for $\theta$ with respect to a $(1,0)$-frame $(Z_\alpha)$. Then the $\CC^n$-valued $1$-form $(\tilde\theta^\alpha)$ defined by
\begin{equation*}
\tilde\theta^\al = e^f(\theta^\al + v^\al\theta) \;,
\end{equation*}
is an admissible coframe for $\tilde\theta$, where 
$v^\al$ is defined by 
\beg\label{e;v}
f_\al = iv_\al + p_{\al\be} \,v^\be.
\eeg
where $f_\alpha = Z_\alpha f$ and $v_\alpha = v^{\bar\beta}g_{\alpha\bar\beta}$. Moreover, if we denote the Levi form and non-partial-integrbility for $\tilde\theta$ with respect to $(\tilde\theta^\al)$ by $(\tilde g_{\al\bar\be})$ and $(\tilde p_{\al\be})$, then 
\begin{equation*}
\tilde g_{\al\bar\be} = g_{\al\bar\be}, \quad \tilde p_{\al\be} = p_{\al\be}.
\end{equation*}
\end{prop}

The last argument of the proposition means that Equation~\ref{e;str1} for $\tilde\theta=e^2f$ is of the form
\begin{equation*}
d\tilde\theta =
 2ig_{\alpha\bar\beta}\,\tilde\theta^\alpha\wedge\tilde\theta^{\bar\beta}
 +p_{\alpha\beta}\,\tilde\theta^\alpha\wedge\tilde\theta^\beta
 +p_{\bar\alpha\bar\beta}\,\tilde\theta^{\bar\alpha}\wedge\tilde\theta^{\bar\beta}
 \;.
\end{equation*}

\subsection{The pseudo-Hermitian strucrure equation} 
Let $(M, \theta)$ be a pseudo-Hermitian manifold and $(\theta^\alpha)$ be an admissible coframe. The pseudo-Hermitian connection form  $(\ind{\omega}{\beta}{\alpha}{})$ for the Hermitian metric $(g_{\alpha\bar\beta})$ in \eqref{e;str1} is defined as follows.

\begin{prop}[\cite{JooLee2015}]\label{p;str2}
The pseudo-Hermitian connection form $(\ind{\omega}{\beta}{\alpha}{})$ is uniquely determined by following equations:
\begin{gather}
d\theta^\alpha
	=\theta^\beta\wedge\ind{\omega}{\beta}{\alpha}{}
	+\ind{T}{\beta}{\alpha}{\gamma}\theta^\beta\wedge\theta^\gamma
	+\ind{N}{\bar\beta}{\alpha}{\bar\gamma}\theta^{\bar\beta}\wedge\theta^{\bar\gamma}
	+\ind{A}{}{\alpha}{\bar\beta} \theta\wedge\theta^{\bar\beta}
	+\ind{B}{}{\alpha}{\beta}{}\theta\wedge\theta^\beta \;, 
	\label{e;str2} 
	\\
dg_{\alpha\bar\beta} 
	= \omega_{\alpha\bar\beta} + \omega_{\bar\beta\alpha} \;, 
	\label{e;comp1} 
	\\
\ind{T}{\beta}{\alpha}{\gamma}
	=-\ind{T}{\gamma}{\alpha}{\beta} \;, 
	\quad 
	\ind{N}{\bar\beta}{\alpha}{\bar\gamma}
	=-\ind{N}{\bar\gamma}{\alpha}{\bar\beta}\;, 
	\quad 
	B_{\al\bar\be} 
	= B_{\bar\be\al} \;.
	\label{e;comp2}
\end{gather}
\end{prop}

Then the connection form $(\ind{\omega}{\beta}{\alpha}{})$ gives the covariant derivative to tensor fields by
\begin{equation}\label{eqn:covariant derivative}
\begin{aligned}
\nabla Z_\alpha 
	&= \ind{\omega}{\alpha}{\beta}{}\otimes Z_\beta \;, 
	& \nabla Z_{\bar\alpha}
	&= \ind{\omega}{\bar\alpha}{\bar\beta}{}\otimes Z_{\bar\beta} \;,
	\\
\nabla \theta^\alpha 
	&= -\ind{\omega}{\beta}{\alpha}{}\otimes\theta^\beta \;,
	& \nabla \theta^{\bar\alpha} 
	&= -\ind{\omega}{\bar\beta}{\bar\alpha}{}\otimes\theta^{\bar\beta}
	\;.
\end{aligned}
\end{equation}
Equations \eqref{e;str1} and \eqref{e;str2} are called the structure equations, and \eqref{e;comp1} and \eqref{e;comp2} are called the compatibility conditions. Tensor coefficients $p_{\alpha\beta}$, $\ind{T}{\beta}{\alpha}{\gamma}$, $\ind{N}{\bar\beta}{\alpha}{\bar\gamma}$, $\ind{A}{}{\alpha}{\bar\beta}$,  $\ind{B}{}{\alpha}{\beta}{}$ are called the torsion coefficients. 

We define the curvature $2$-form by 
\begin{equation*}
\ind{\Theta}{\beta}{\alpha}{} = d\ind{\omega}{\beta}{\alpha}{} - \ind{\omega}{\beta}{\gamma}{}\wedge\ind{\omega}{\gamma}{\alpha}{} \;.
\end{equation*}
The coefficients of pseudo-Hermitian curvature is defined by 
\begin{equation*}
\ind{R}{\beta}{\alpha}{\gamma\bar\sigma}
	=2\ind{\Theta}{\beta}{\alpha}{}(Z_\gamma, Z_{\bar\sigma}) \;.
\end{equation*}

Let $\tilde\theta= e^{2f}\theta$ be a pseudoconformal change of the pseudo-Hermitian structure $(M, \theta)$ and let $(\tilde\theta^\alpha)$ the admissible coframe for $\tilde\theta$ defined as in Proposition \ref{p;conformal change}.

\begin{prop}[\cite{JooLee2015}]\label{p;change} The coefficieints of torsion and curvature of $\theta$ and $\tilde\theta$ with respect to $(\theta^\al)$ and $(\tilde\theta^\al)$ are related as follows.
\begin{align*}
\ind{\widetilde T}{\beta}{\alpha}{\gamma}
 	&=
 	e^{-f}\paren{
 	\ind{T}{\beta}{\alpha}{\gamma}
 	+v^\alpha p_{\beta\gamma}
 	-p_{\gamma\rho}v^\rho\ind{\delta}{}{\alpha}{\beta}
 	+p_{\beta\rho}v^\rho\ind{\delta}{}{\alpha}{\gamma}
 	} \;,
 	\\
\ind{\widetilde N}{\bar\beta}{\alpha}{\bar\gamma}
 	&=
 	e^{-f}\paren{
 	\ind{N}{\bar\beta}{\alpha}{\bar\gamma}
	+v^\alpha p_{\bar\beta\bar\gamma}
 	} \;,
 	\\
\ind{\widetilde A}{}{\alpha}{\bar\beta}
 	&=
 	e^{-2f}\paren{
 	\ind{A}{}{\alpha}{\bar\beta}
 	-g^{\alpha\bar\gamma}v_{\bar\gamma;\beta}
 	-2iv^\alpha v_{\bar\beta}
 	-2v^{\bar\gamma}\ind{N}{\bar\gamma}{\alpha}{\bar\beta}
 	-2v^\alpha v^{\bar\gamma}p_{\bar\gamma\bar\beta}
 	} \;,
 	\\
\ind{\widetilde B}{}{\alpha}{\beta}
 	&=
 	e^{-2f}\Big(
 	\ind{B}{}{\alpha}{\beta}
 	-\frac{1}{2}(g^{\alpha\bar\gamma}v_{\bar\gamma;\beta}+v_{\beta;\bar\gamma}g^{\alpha\bar\gamma})
 	+\ind{\delta}{}{\alpha}{\beta}f_0
 	\\
 	&\qquad\qquad
 	-v^\alpha v^\gamma p_{\gamma\beta}
 	-v_\beta v^{\bar\gamma}\ind{p}{\bar\gamma}{\alpha}{}
 	-v^\gamma\ind{T}{\gamma}{\alpha}{\beta}
 	-v^{\bar\gamma}\ind{T}{\bar\gamma\beta}{\alpha}{}
 	\Big) \;,
	\\
\ind{\widetilde R}{\beta}{\alpha}{\lambda\bar\mu}
	&=
 	e^{-2f}\Big(
 	\ind{R}{\beta}{\alpha}{\lambda\bar\mu}
 	-\ind{\delta}{\beta}{\alpha}{}f_{\lambda\bar\mu}
 	-\ind{\delta}{\beta}{\alpha}{}f_{\bar\mu\lambda}
 	+2ig_{\beta\bar\mu}\ind{v}{}{\alpha}{;\lambda}
 	-2iv_{\beta\bar\mu}\ind{\delta}{}{\alpha}{\lambda}
	 \\
  	&\qquad+i(g^{\alpha\bar\gamma}v_{\bar\gamma;\beta}+v_{\beta;\bar\gamma}g^{\alpha\bar\gamma})g_{\lambda\bar\mu}
	-4\ind{\delta}{\beta}{\alpha}{}g_{\lambda\bar\mu}v^\gamma v_\gamma
	-4g_{\beta\bar\mu}\ind{\delta}{}{\alpha}{\lambda}v^\gamma v_\gamma
 	\\
  	&\qquad-2ip_{\beta\gamma}v^\gamma v^\alpha g_{\lambda\bar\mu}
  	-2iv_\beta v^{\bar\gamma} \ind{p}{\bar\gamma}{\alpha}{}g_{\lambda\bar\mu}
  	+2iv^\gamma\ind{T}{\gamma}{\alpha}{\beta}g_{\lambda\bar\mu}
  	-2iv^{\bar\gamma}\ind{T}{\bar\gamma\beta}{\alpha}{}g_{\lambda\bar\mu}
 	\Big) \;.
	\\
\ind{\widetilde R}{\beta}{\alpha}{\gamma\bar\sigma}
	&=
 	e^{-2f}\Big(
 	\ind{R}{\beta}{\alpha}{\gamma\bar\sigma}
 	-\ind{\delta}{\beta}{\alpha}{}f_{\gamma\bar\sigma}
 	-\ind{\delta}{\beta}{\alpha}{}f_{\bar\sigma\gamma}
 	+2ig_{\beta\bar\sigma}\ind{v}{}{\alpha}{;\gamma}
 	-2iv_{\beta\bar\sigma}\ind{\delta}{}{\alpha}{\gamma}
	 \\
  	&\qquad+i(g^{\alpha\bar\lambda}v_{\bar\lambda;\beta}+v_{\beta;\bar\lambda}g^{\alpha\bar\lambda})g_{\gamma\bar\sigma}
	-4\ind{\delta}{\beta}{\alpha}{}g_{\gamma\bar\sigma}v^\lambda v_\lambda
	-4g_{\beta\bar\sigma}\ind{\delta}{}{\alpha}{\gamma}v^\lambda v_\lambda
 	\\
  	&\qquad-2ip_{\beta\lambda}v^\lambda v^\alpha g_{\gamma\bar\sigma}
  	-2iv_\beta v^{\bar\lambda} \ind{p}{\bar\lambda}{\alpha}{}g_{\gamma\bar\sigma}
  	+2iv^\lambda\ind{T}{\lambda}{\alpha}{\beta}g_{\gamma\bar\sigma}
  	-2iv^{\bar\lambda}\ind{T}{\bar\lambda\beta}{\alpha}{}g_{\gamma\bar\sigma}
 	\Big) \;.
\end{align*}
Moreover, the coefficients of covariant derivatives of $p_{\al\be}$ change as follows.
\begin{align*}
\tilde p_{\alpha\beta;\gamma}
 	&=
 	e^{-f}\paren{
	p_{\alpha\beta;\gamma}
	-2p_{\alpha\beta}f_\gamma
	-2ip_{\alpha\gamma}v_\beta
	-2ip_{\gamma\beta}v_\alpha
	} \;,
	\\
\tilde p_{\alpha\beta;\bar\gamma}
	&=
	e^{-f}\paren{
	p_{\alpha\beta;\bar\gamma}
	+2p_{\alpha\beta}f_{\bar\gamma}
	-2ip_{\alpha\lambda}v^\lambda g_{\beta\bar\gamma}
	-2ip_{\lambda\beta}v^\lambda g_{\alpha\bar\gamma}
	} \;.
\end{align*}
\end{prop}

Here  $v_{\bar\gamma;\beta}$, $v_{\beta;\bar\gamma}$ $p_{\alpha\beta;\gamma}$, $p_{\alpha\beta;\bar\gamma}$ are coefficients of covariant derivatives of the tensors $(v_\alpha)$, $(p_{\alpha\beta})$ as defined in \eqref{eqn:covariant derivative}:
\begin{align*}
v_{\bar\gamma;\beta}
	&= Z_\beta v_{\bar\gamma} - \ind{\omega}{\bar\gamma}{\bar\sigma}{}(Z_\beta)v_{\bar\sigma}
	\;,
	\\
v_{\beta;\bar\gamma} 
	&= Z_{\bar\gamma}v_\beta - \ind{\omega}{\beta}{\sigma}{}(Z_{\bar\gamma})v_\sigma
	\;,
	\\
p_{\alpha\beta;\gamma} 
	&= Z_\gamma p_{\alpha\beta}
	-\ind{\omega}{\alpha}{\sigma}{}(Z_\gamma)p_{\sigma\beta}
	-\ind{\omega}{\beta}{\sigma}{}(Z_\gamma)p_{\alpha\sigma} \;,
	\\
p_{\alpha\beta;\bar\gamma} 
	&= Z_{\bar\gamma} p_{\alpha\beta}
	-\ind{\omega}{\alpha}{\sigma}{}(Z_{\bar\gamma})p_{\sigma\beta}
	-\ind{\omega}{\beta}{\sigma}{}(Z_{\bar\gamma})p_{\alpha\sigma} \;.
\end{align*}


\subsection{Generalized Heiseneberg groups}\label{subsec:models}
Let $(z,t)=(z^1,...,z^n,t)$ be the standard coordinates on $\CC^n\times\RR$ and let $P = (P_{\al\be})$ be a constant complex skew-symmetric matrix of size $n\times n$. Let 
\beg\label{e;10vector}
Z_\al = \frac{\partial}{\partial z^\al} + \left(iz^{\bar\al} + P_{\al\be}z^\be\right)\frac{\partial}{\partial t}
\eeg
for $\al=1,...,n$. The almost CR structure whose $H_{1,0}$ space is spanned by $\{Z_\al\}$ will be denoted by $J_P$. We define the generalized Heisenberg group corresponding to $P$ by the almost CR manifold $(\CC^n\times\RR, J_P)$ and will denote it by $\mh_P$. In case $P=0$, $\mh_0$ is the classic Heisenberg group. Generally, we have

\begin{prop}[\cite{JooLee2015}]\label{p;heisenberg}
Let $* = *_P$ be the binary operation on $\mh_P$ defined by 
$$(z,t)*(z', t') = \left(z+z', t+t' + 2\mathrm{Im}\,(z^\al z'^{\bar\al}) - 2\mathrm{Re}\, (P_{\al\be}z^\al z'^{\be})\right).$$
Then this operation makes $\mh_P$ a Lie group and $Z_\al$ defined in \eqref{e;10vector} is left invariant under $*$. In particular, $\mh_P$ is a homogeneous almost CR manifold. 
\end{prop}

Since the canonical $(1,0)$-vector fields $Z_1,\ldots,Z_n$ of $J_P$ satisfy
\begin{equation}\label{eqn:bracket}
[Z_\alpha,Z_{\bar\beta}]=-2i\delta_{\alpha\bar\beta}\pd{}{t}
 \; ,
 \quad
[Z_\alpha,Z_\beta]=-2P_{\alpha\beta}\pd{}{t}
 \;,
\end{equation}
the CR structure of $\mh_P$ is integrable if and only if $P=0$.  Each model $\mh_P$ admits the \emph{dilation} $D_\tau(z,t)=(e^\tau z,e^{2\tau}t)$ as a CR automorphism. Therefore the Heisenberg group model has a CR-contraction at the origin $0$ so at every point of $\mh_P$ by the homogeneity.

Let $\theta_P$ be a real $1$-form on $\mh_P$ defined by
\begin{equation}\label{e;thetap}
\theta_P = dt+ iz^\al dz^{\bar\al} - iz^{\bar\al}dz^\al + P_{\al\be}z^\al dz^\be + P_{\bar\al\bar\be}z^{\bar\al}dz^{\bar\be}.
\end{equation}
Note that $\theta_P(Z_\al) =0$ for every $\al=1,...,n$. Moreover, 
\begin{equation*}
d\theta_P = 2i dz^\al\w dz^{\bar\al} + P_{\al\be}dz^\al\w dz^\be + P_{\bar\al\bar\be}dz^{\bar\al}\w dz^{\bar\be}.
\end{equation*}
Therefore, we see that $\mh_P$ is a strongly pseudoconvex almost CR manifold and $(\theta^\al = dz^\al)$ is an admissible coframe for $\theta$ which is dual to $(Z_\al)$. Since $d\theta^\al = ddz^\al =0$, we see that $(\mh_P, \theta_P)$ has a vanishing pseudo-Hermitian connection $(\ind{\omega}{\beta}{\alpha}{})$. Moreover all coefficients of torsion  and curvature, except $p_{\alpha\beta}$, vanish identically and $p_{\alpha\beta}\equiv P_{\alpha\beta}$. Conversely,

\begin{prop}[\cite{JooLee2015}]\label{p;equiv}
Let $(M,\theta)$ be a pseudo-Hermitian manifold. If 
\begin{equation*}
p_{\alpha\beta;\gamma}
 \equiv\ind{T}{\beta}{\alpha}{\gamma}
 \equiv\ind{N}{\bar\beta}{\alpha}{\bar\gamma}
 \equiv\ind{A}{}{\alpha}{\bar\beta}
 \equiv\ind{B}{}{\alpha}{\beta}
 \equiv\ind{R}{\beta}{\alpha}{\lambda\bar\mu}
 \equiv0
\end{equation*}
for some (and hence for all) admissible coframe, then $(M,\theta)$ is locally equivalent to $(\mh_P, \theta_P)$ as a pseudo-Hermitian manifold.
\end{prop}


\subsection{CR mappings and diffeomorphisms of Heisenberg models}
As we mentioned, except the standard Heisenberg group $\mh_0$, every CR structure of $\mh_P$ is non-integrable. The non-integrability of the structure gives some restriction to the CR mappings. The following lemma for this observation will be used in the proof of Theorem~\ref{t;thm2}.

\begin{lem}[cf. Proposition~3.3 in \cite{JooLee2015}]\label{lemma:CR mapping}
Let $\mh_P$ and $\mh_{P'}$ be non-integrable Heisenberg group models of the same dimension. If there is a local CR diffeomorphism
\begin{equation*}
G(z,t) = (w,s)=(w^1,\ldots,w^n,s)
\end{equation*}
from $\mh_P$ to $\mh_P'$, then 
\begin{itemize}
\item[$(1)$]  each $w^\lambda$ ($\lambda=1,\ldots,n$) is indepedent of $t$-variable and holomorphic in $z$-variables;
\item[$(2)$] $s(t)=rt+c(z)$ for some constant $r$ and real-valued function $c$;
\item[$(3)$] the constant $r$ is determined by
\begin{equation*}
r=\frac{1}{n}\sum_{\alpha,\beta=1}^n \abs{\pd{w^\beta}{z^\alpha}}^2 \;.
\end{equation*}
\end{itemize}
\end{lem}

\begin{proof}
We will use $(w,s)=(w^1,\ldots,w^n,s)$ as a standard coordinates of $\mh_{P'}$ also. Let $K_\alpha(z)= iz^{\bar\alpha}+\mrt_{\alpha\beta}z^\beta$ and $K'_\lambda(w)=iw^{\bar\lambda}+P'_{\lambda\mu}w^\mu$. Then the canonical $(1,0)$-frames $(Z_\alpha)$ and $(Z'_\lambda)$ of $\mh_P$ and $\mh_{P'}$ can be determined by
\begin{equation*}
Z_\alpha=\pd{}{z^\alpha}+K_\alpha\pd{}{t}
\quad\text{and}\quad
Z'_\lambda=\pd{}{w^\lambda}+K'_\lambda\pd{}{s}
 \;,
\end{equation*}
respectively. Let us consider
\begin{equation*}
dF(Z_\alpha) = (Z_\alpha w^\lambda)\pd{}{w^\lambda} +(Z_\lambda w^{\bar\mu})\pd{}{w^{\bar\mu}} + (Z_\alpha s)\pd{}{s} \;.
\end{equation*}
Since $F$ is a CR mapping, we have $dF(Z_\alpha)=\ind{a}{\alpha}{\lambda}{}Z'_\lambda$ for some functions $\ind{a}{\alpha}{\lambda}{}$. Comparing two expressions of $dF(Z_\alpha)$, we can conclude that $\ind{a}{\alpha}{\lambda}{} = Z_\alpha w^\lambda$ so
\begin{equation}\label{5-conclusion}
dF(Z_\alpha)
 =
 (Z_\alpha w^\lambda)Z'_\lambda
 \;,
 \quad
Z_\alpha w^{\bar\mu}=0
 \;, \quad
 Z_\alpha s=(Z_\alpha w^\lambda)K'_\lambda\circ F
 \;.
\end{equation}

\medskip

From the non-integrability of $\mh_P$, we can choose $\alpha,\beta$ so that $P_{\alpha\beta}\neq0$. Then we have
\begin{equation}\label{eqn:first commutator}
[dF(Z_\alpha),dF(Z_\beta)] = dF([Z_\alpha,Z_\beta]) = -2\mrt_{\alpha\beta}dF\paren{\pdl{}{t}}
\end{equation}
from \eqref{eqn:bracket}. One can see that the local vector field 
\begin{equation*}
[dF(Z_\alpha),dF(Z_\beta)] = [(Z_\alpha w^\lambda)Z'_\lambda,(Z_\beta w^\mu)Z'_\mu]
\end{equation*}
on $\mh_{P'}$ has no terms in $\pdl{}{w^{\bar\mu}}$. From
\begin{equation*}
-2P_{\alpha\beta}dF\paren{\pdl{}{t}}
= -2P_{\alpha\beta}\paren{
	\pd{w^\lambda}{t}\pd{}{w^\lambda}
	+\pd{w^{\bar\mu}}{t}\pd{}{w^{\bar\mu}}
	+\pd{s}{t}\pd{}{s}
	}
\end{equation*}
we have
\begin{equation*}
\pd{w^{\bar\mu}}{t}=0 \quad\text{so}\quad \pd{w^{\mu}}{t}=0
\end{equation*}
for each $\mu=1,\ldots,n$. Simultaneously, we have
\begin{equation*}\label{6-conclusion}
Z_\alpha w^\lambda=\pd{w^\lambda}{z^\alpha} \; , \quad Z_{\bar\alpha}w^\lambda=\pd{w^\lambda}{z^{\bar\alpha}}=0 
\end{equation*}
from \eqref{5-conclusion}. This implies that each $w^\lambda$ is independent of $t$-variable and holomorphic in $z$-variables. Indeed $w^\alpha$ is defined on the open set $\cU_1=\pi_1(\cU)$ in $\CC^n$ where $\cU$ is the domain of $F$ in $\CC^n\times\RR$ and $\pi_1:\CC^n\times\RR\to\CC^n$ is the natural projection.  When we write
\begin{equation*}
[dF(Z_\alpha),dF(Z_\beta)]  
	= \left[ \pd{w^\lambda}{z^\alpha}Z'_\lambda,\pd{w^\mu}{z^\beta}Z'_\mu\right] 
	= \ind{b}{\alpha\beta}{\lambda}{}Z'_\lambda
	-2P'_{\lambda\mu} \pd{w^\lambda}{z^\alpha}\pd{w^\mu}{z^\beta}\pd{}{s}
\end{equation*}
by some functions $\ind{b}{\alpha\beta}{\lambda}{}$, applying $dF\paren{\pdl{}{t}} = (\pdl{s}{t})\pdl{}{s}$ to \eqref{eqn:first commutator} we have that $\ind{b}{\alpha\beta}{\lambda}{}=0$ for any $\lambda$ so that
\begin{equation*}\label{6-c}
\pd{s}{t}
 =\frac{P'_{\lambda\mu}}{P_{\alpha\beta}}(\partial_\alpha w^\lambda)(\partial_\beta w^\mu)
 \;.
\end{equation*}
Note that $P_{\alpha\beta}\neq 0$ by the choice of $\alpha,\beta$. This implies that $\pdl{s}{t}$ is independent of $t$-variable so $s(z,t)=r(z)t+c(z)$ for some smooth real-valued functions $r=\pdl{s}{t}$ and $c$ defined on $\cU_1$.  The third equation of \eqref{5-conclusion} can be written as
\begin{equation*}
 Z_\alpha s
 	= \pd{s}{z^\alpha}+K_\alpha\pd{s}{t}= \pd{w^\lambda}{z^\alpha} K'_\lambda\circ F
\end{equation*}
Since $K_\alpha$, $\pdl{s}{t}$, $\pdl{w^\lambda}{z^\alpha}$, $K'_\lambda\circ F$ are all independent of $t$-variable, so is $\pdl{s}{z^\alpha}$. This means that $\pdl{r}{z^\alpha}=0$ for each $\alpha$ because $\pdl{s}{z^\alpha} = (\pdl{r}{z^\alpha})t+\pdl{c}{z^\alpha}$. Hence the real-valued function $r$ is a constant: $s(z,t)=rt+c(z)$.

\medskip

From $[dF(Z_\alpha),dF(Z_{\bar\beta})]=dF([Z_\alpha,Z_{\bar\beta}])=-2i\delta_{\alpha\bar\beta}dF(\partial_t)$, one can easily get
\begin{equation*}
\delta_{\alpha\bar\beta}\pd{t}{s}
 =
 \delta_{\lambda\bar\mu}\pd{w^\lambda}{z^\alpha}\pd{w^{\bar\mu}}{z^{\bar\beta}}
 \;, \quad
 \text{so}\quad
 n\pd{t}{s}=\sum_{\alpha,\lambda}\abs{\pd{w^\lambda}{z^\alpha}}^2
 \; .
\end{equation*}
This completes the proof.
\end{proof}

The standard Heisenberg group $\mh_0$ is CR-equivalent to $S^{2n+1}\setminus\{p\}$, the unit sphere minus one point. Therefore $\mh_0\setminus\{0\}$ is also equivalent to $S^{2n+1}\setminus\{p,-p\}$. By the symmetry of $S^{2n+1}$, $\mh_0\setminus\{0\}$ admits a non-trivial involution CR automorphism: 
\begin{equation*}
(z,t)\mapsto \paren{\frac{z}{-\abs{z}^2+it},\frac{-t}{\abs{z}^4+t^2}} \;.
\end{equation*}
But for a non-integrable Heisenberg model $\mh_P$, every CR automorphism of $\mh_P\setminus\{0\}$ is extended to a global CR automorphism of $\mh_P$ (Proposition~3.3 in \cite{JooLee2015}).


As an application of Lemma~\ref{lemma:CR mapping}, we can describe the CR automorphism group of the Heisenberg group model (see \cite{JooLee2015}).

\begin{thm}\label{thm:CR automorphism group}
Let $\mh_P$ be a Heisenberg group model.
\begin{itemize}
\item[$(1)$] $\mh_P$ is CR equivalent to a Heisenberg group model $\mh_{P'}$ if and only if there is a unitary matrix $U=(\ind{U}{\beta}{\alpha}{})\in\UG(n)$ with $\mrt_{\alpha\beta}
    =\ind{U}{\alpha}{\lambda}{}\mrt'_{\lambda\mu}\ind{U}{\beta}{\mu}{}$.
\item[$(2)$] The isotropy group $\Aut_0(\mh_P)=\set{F\in\Aut(\mh_P):F(0)=0}$
    can be composed by
    \begin{equation*}
    \Aut_0(\mh_P)=\Aut_0(\mh_P,\theta_P)\oplus\{D_\tau\}
    \end{equation*}
    where $\Aut_0(\mh_P,\theta_\mrt)$ is the pseudo-Hermitian isotropy group.
\item[$(3)$] If $\mrt=0$ equivalently $J_P$ is integrable, the pseudo-Hermitian isotropy group $\Aut_0(\mh_0,\theta_0)$ is isomorphic to the unitary group $\UG(n)$. If $\mrt\neq0$, then
    \begin{equation*}
    \Aut_0(\mh_P,\theta_P)\simeq\set{U\in\UG(n):U^tP U=P}
    \; .
    \end{equation*}
More precisely, every element of $\Aut_0(\mh_P,\theta_P)$ of the form:
\begin{equation*}
(z,t)\to (U(z),t) \;.
\end{equation*}
for some $U\in\UG(n)$.
\end{itemize}
\end{thm}


\section{Contracting CR automorphism}\label{sec:contractingCRa}
In this section, we introduce the CR contraction and its stable manifold and show that there is a canonical contact form of the CR contaction.

\subsection{The CR contractions and the stable manifolds}
Let $(M, \theta)$ be a paeudo-Hermitian manifold. We say that a CR automorphism $\vp$ of $M$ is \emph{(weakly) contracting} at $o\in M$ if 
$$\vp^*\theta |_o = \mu\,\theta|_o$$ for some real $0<\mu<1$. This definition is independent of choice of contact form. Let $(\theta^\al)$ is an admissible coframe for $\theta$ satisfying 
\begin{equation*}
g_{\al\bar\be}(o) = \delta_{\al\bar\be} \;.
\end{equation*}
Since $\varphi$ is preserving the CR structure, we can write
\beg\label{e;transform}
\vp^*\theta^\al = a^\al_\be \theta^\be + c^\al \theta,
\eeg
for some functions $a^\alpha_\beta$ and $c^\alpha$. Then from \eqref{e;str1}, we see that 
\begin{equation*}
\sum_{\gamma}a^\ga_\al a^{\bar\ga}_{\bar\be} = \mu\delta_{\al\bar\be} \quad\text{at $o$.}
\end{equation*}
Therefore, the matrix $A = (a^\al_\be)$ is a normal operator (hence, diagonalizable) and its eigenvalues have modulus $\sqrt\mu$ at $o$. By a unitary change of the frame, we may assume  
$$A = (a^\al_\be) = \begin{pmatrix} \la_1&0&\cdots&0\\
                                                       0&\la_2&\cdots&0\\
							\vdots& &\ddots&\vdots\\
							0&\cdots&0&\lambda_n
\end{pmatrix}$$
at $o$, where $|\la_\al|= \sqrt\mu$ for $\al=1,...,n$. 

\begin{prop}\label{p;normal}
There exists a contact form $\theta$ such that $c^\al$ in \eqref{e;transform} vanishes at $o$.
\end{prop}

\begin{proof}
Let $\theta$ be any pseudo-Hermitian structure on $M$. Assume again that $(\theta^\al)$ be an admissible  coframe for $\theta$ such that $g_{\al\bar\be}(o)=\delta_{\al\bar\be}$ and $A=(a^\al_\be)$ in \eqref{e;transform}  is diagonal at $o$. Let $\tilde\theta = e^{2f}\theta$ be a pseudoconformal change of $\theta$, where $f$ is a real-valued smooth function with $f(o) =0$ and let $(\tilde\theta^\alpha)$ be its admissible coframe defined by $\tilde\theta^\alpha= e^f(\theta^\alpha+v^\alpha\theta)$ as in Proposition~\ref{p;conformal change}. When we let $\varphi^*\tilde\theta^\alpha=\tilde{a}_\beta^\alpha\tilde\theta^\beta+\tilde{c}^\alpha\tilde\theta$, it follows 
that  $\tilde{a}_\beta^\alpha\equiv a_\beta^\alpha$ and 
\begin{equation*}
\tilde{c}^\alpha=c^\alpha-(a^\alpha_\beta -\mu\delta^\alpha_\beta) v^\beta \quad\text{at $o$.}
\end{equation*}
We want to have 
\begin{equation*}
\vp^*\tilde\theta^\al = a^\al_\be \tilde\theta^\be \quad\text{at $o$.}
\end{equation*}
From \eqref{e;transform} and Proposition \ref{p;conformal change}, $v=(v^\al(o))$ must be a solution of
\begin{equation*}
(A - \mu I)v = c,
\end{equation*}
where $c = (c^\al(o))$. Note that the matrix $A-\mu I$ is invertible at $o$ since the eigenvalues of $A$ have modulus $\sqrt\mu >\mu$. Therefore, the above equation is uniquely solved for given $c$, and by taking $f$ whose derivative $(f_\alpha)$ at $o$ gives rise to the solution $v = (v^\al(o))$ of \eqref{e;v}, we see that $\tilde\theta$ satisfies the condition of this proposition.
\end{proof}

From this proposition, we may assume that 
\beg\label{e;f^*}
\vp^*|_o = \begin{pmatrix} \mu&0&\cdots&0\\
                                                       0&\la_1&\cdots&0\\
							\vdots& &\ddots&\vdots\\
							0&\cdots&0&\lambda_n
\end{pmatrix}
\eeg
with respect to $\theta, \theta^\al$.

Note that $TM = H\oplus \mathbb R T$. With regard to this decomposition, we can define a Riemannian metric $ds^2$ on $M$ so that 
\begin{equation*}
ds^2 (T, T) =1,\quad ds^2 (T, X) =0,\quad ds^2(X, Y) = d\theta (X, JY)
\end{equation*}
for any $X, Y\in H$. For $x\in M$, we denote by $|x|$ the geodesic distance from $o$ to $x$.  Then \eqref{e;f^*} implies that 
\beg\label{e;basic}
|\vp(x)| \leq \eta |x|
\eeg
for some $0<\eta<1$ if $x$ is in a sufficiently small neighborhood $U$ of $o$. Let $\W$ be the \emph{stable manifold} of $\vp$. That is, $x\in \W$ if and only if there exists a neighborhood $V$ of $x$ such that $\vp^k \rightarrow o$ uniformly on $V$ as $k\rightarrow \infty$. From \eqref{e;basic}, it turns out $\W$ is a nonempty open subset of $M$. In fact, it can be seen easily that 
$$\W = \bigcup_{k\in\mathbb{Z}} \vp^k (U).$$

\subsection{The canonical contact form of the CR contraction}\label{subsec:canonical contact form} Let $M$ be a strongly pseudoconvex almost CR manifold and $\varphi$ be a CR contraction at $o\in M$. We have $0<\mu<1$ with $\varphi^*\theta|_o=\mu\theta|_o$ for any contact form $\theta$. If a contact form $\theta$ on the stable manifold $\W$ satisfies
\begin{equation}\label{e;canonical}
\varphi^*\theta=\mu\theta,
\end{equation}
we call $\theta$ a \emph{canonical contact form} of $\varphi$.

\begin{prop}\label{p;canonical}
There exists a unique (up to constant multiple) continuous canonical contact form $\theta$ on $\W$.
\end{prop}

\begin{proof}
We first show the uniqueness. Suppose that $\theta_1$ and $\theta_2$ are two continuous contact form satisfying \eqref{e;canonical} and that $\theta_1|_o = \theta_2|_o$. Then $\theta_1 - \theta_2 = u\theta_1$ for some continuous function $u$ on $\W$ with $u(o)=0$. Then for any $x\in \W$,
$$u(\vp(x))\mu \theta_1|_x = \vp^*(u\theta_1)|_x = \vp^*(\theta_1-\theta_2)|_x = \mu(\theta_1-\theta_2)|_x= u(x)\mu\theta_1|_x.$$ 
Therefore, we have $u(\vp(x)) = u(x)$ for every $x\in \W$. Conesequently,
$$u(x) = u(\vp^k(x))\rightarrow u(o) =0$$ 
as $k\rightarrow\infty$ for every $x\in \W$. 

\medskip

Let $\tilde\theta$ be any smooth contact form on $M$. Let 
$$\theta_k = \frac1{\mu^k}\, (\vp^k)^* \tilde\theta$$
for $k=1,2,...$. If we denote 
$$\theta_k = u_k \tilde\theta,$$ then it turns out that
$$u_k(x) = \frac1{\mu^k}\,v(\vp^{k-1}(x))\cdots v(x) = \prod_{j=1}^k a_j(x),$$
where $\vp^*\tilde\theta = v\tilde\theta$ and $a_j(x) = v(\vp^{j-1}(x))/\mu$. Since $\vp^k \rightarrow o$ locally uniformly on $\W$, it suffices to show that $u_k$ converges uniformly on $U$ to guarantee the convergence of $\theta_k$ on $\W$. Note that the infintie product $\prod a_j(x)$ converges absolutely and uniformly on $U$ if so is the infinite series $\sum |a_j(x) -1 |$. Since 
$0< v(x)\leq \mu + C |x|$ for some constant $C$ in $U$, we see that 
$$0<a_j(x)\leq 1 + C|\vp^{j-1}(x)| \leq 1 + C|x| \eta^{j-1}$$
on $U$, for some constant $ C>0$ by \eqref{e;basic}.  Since $0<\eta<1$, we conclude that 
$\sum_j |a_j-1| \leq C|x| \sum_j \eta^{j-1}$ converges uniformly on $U$. This yields that $u_k \rightarrow u$ locally uniformly on $\W$ for some positive continuous function $u$ on $\W$. It is obvious that 
$\theta = u\tilde\theta$ satisfies \eqref{e;canonical}.
\end{proof}

The next proposition implies that the canonical contact form $\theta$ in Proposition~\ref{p;canonical} is smooth.

\begin{prop}\label{p;smooth}
A canonical contact form $\theta$ induced by a contracting CR automorphism $\vp$ is indeed $C^\infty$-smooth on $\W$. 
\end{prop}

\begin{proof}
We will prove for every positive integer $s$, $\{D^{(s)} u_k : k\geq0\}$ is locally uniformly bounded on $\W$, where $D^{(s)}$ represents the $s$-th order differential operator. We assume $\tilde\theta$ and $\theta^\al$ were chosen such that $\vp^*$ has the form of \eqref{e;f^*} at $o$. Let $T$ be the characteristic vector field for $\tilde\theta$ and $Z_\al$ be the dual frame for $\{\theta^\al\}$. Then there exists a local coordinates $(z^1,...,z^n, t)\in \CC^n\times\RR$ such that 
$$\left.\frac{\partial}{\partial z^\al}\right|_o = Z_\al|_o,\quad \left.\frac{\partial}{\partial t}\right|_o = T|_o.$$ 
Therefore, the Jacobian of $\vp$ in this coordinate system has the same form with \eqref{e;f^*} at $o$. Shrinking the neighborhood $U$ of $o$ if necessary, we may assume 
\beg\label{e;estimate}
|D\vp| \leq \eta
\eeg
on $U$ for some $0<\eta<1$, where $D$ is the differential operator in this coordinate system.  

\medskip

{\em Claim.} For each positive integer $s$, there exists a polynomial $P_s$ depending on $s$ and $\|\vp\|_{(s)}$ such that 
$$|D^{(s)} \vp^j| \leq P_s(j) \eta^{j-s+1} $$
on $U$ whenever $j\geq s$, where $\|\cdot\|_{(s)}$ denotes the $C^s(U)$-norm. 

\medskip

Assume for a while that this claim is true. Since $a_j = v(\vp^{j-1})/\mu$,  if $j\geq s$, then
\beg\label{e;est}
|D^{(s)} a_j| = \mu^{-1}|D^{(s)}(v(\vp^{j-1}))|\leq \wt P_s(j)\eta^{j-s}
\eeg
on $U$ for some polynomial $\wt P_s$ depending on $\|v\|_{(s)}$ and $P_1,...,P_s$,  from the chain rule and Claim.

Now, we finish the proof of this proposition. Recall that $u_k = \prod_{j=1}^k a_j.$ Therefore, 
\beq
D^{(s)} u_k &=& \sum_j D^{(s)} a_j \prod_{l\neq j}a_l + \sum_{j_1\neq j_2} D^{(s-1)}a_{j_1} Da_{j_2}\prod_{l\neq j_1, l\neq j_2}a_l\\
& & + \cdots + {\sum_{j_1,...,j_s}}'Da_{j_1} \cdots Da_{j_s} \prod_{l\neq j_1,...,l\neq j_s} a_l, 
\eeq
where ${\sum}'$ means the summation over mutually distinct indices. Taking $U$ sufficiently small, we may assume $\mu/2 \leq v(x) \leq 2\mu$ for $x\in U$. Then we have $a_j \geq1/2$ on $U$ for every $j$. Since the infinite product $\prod_j a_j$ converges uniformly on $U$, we see that 
$$   \prod_{l\neq j_1,...,l\neq j_s} a_l =  a_{j_1}\cdots a_{j_s} \prod_l a_l \leq 2^s \prod_l a_l \leq C <\infty$$ on $U$ for some constant $C>0$. Therefore from \eqref{e;est},
$$|D^{(s)} u_k| \leq  C \sum_j \wt P_s(j) \eta^{j-s} \leq C' <\infty$$ on $U$ for some constants $C$ and $C'$ independent of $k$. 
This implies that $\{D^{(s)} u_k : k\geq 1\}$ is locally uniformly bounded on $\W$. Therefore, for each $s>1$, we have a subsequence of $\{u_k\}$ convergent to $u$ in local $C^{s-1}$-sense. Since $s$ is arbitrary, we conclude that $u$ is $C^\infty$-smooth. 

\medskip

Now we prove the claim above. In case $s=1$, it turns out easily that 
$$|D\vp^j| \leq \eta^j$$  on $U$ for every $j\geq 1$, from the chain rule and \eqref{e;estimate}. If $s=2$, then 
\beq D^{(2)}\vp^j &=& D(D\vp (\vp^{j-1})\cdot D\vp^{j-1})\\
 &=& D^{(2)}\vp (\vp^{j-1})\cdot D\vp^{j-1}\cdot D\vp^{j-1} + D\vp(\vp^{j-1})\cdot D^{(2)}\vp^{j-1}.\eeq 
Therefore, 
\beg\label{e;rec}
|D^{(2)}\vp^j| \leq \|\vp\|_{(2)} \eta^{2j-2} + \eta |D^{(2)}\vp^{j-1}|\leq \|\vp\|_{(2)}\eta^{j-1} + \eta |D^{(2)}\vp^{j-1}|.
\eeg
for every $j\geq 2$. If $j=2$, then 
$|D^{(2)}\vp^2| \leq 2\|\vp\|_{(2)}\eta.$
Therefore, if we choose $P_2(j) = \|\vp\|_{(2)} j$, the recursive relation \eqref{e;rec} implies that 
$$|D^{(2)}\vp^j| \leq P_2(j) \eta^{j-1}$$ on $U$ for every $j\geq 2$. 

For more general $s>1$, if we have already chosen $P_1\equiv1, P_2,\ldots, P_{s-1}$, then we can also show that 
\beg\label{e;rec2}
|D^{(s)}\vp^j| \leq Q(j) \eta^{j-s+1} + \eta |D^{(s)}\vp^{j-1}|
\eeg
on $U$ for $j\geq s$, where $Q(j)$ is a polynomial in $j$ determined by $\|\vp\|_{(s)}$ and $P_1,...,P_{s-1}$. Therefore, if $j=s$, then $|D^{(s)}\vp^j|\leq C_s \eta$ for some $C_s$ depending on $\|\vp_s\|$ and if we choose $P_s$ a polynomial with degree greater than that of $Q$ such that 
$$P_s(s) \geq C_s,\quad P_s(j) \geq Q(j) + P_s(j-1),$$ 
then the relation \eqref{e;rec2} yields the conclusion. 
\end{proof}


\section{Proofs of main theorems}\label{sec:proof}

In this section, we will prove Theorem~\ref{t;main} and  characterize the ambient manifold $M$ as the standard sphere under further assumption that the contracting automorphism has another fixed point which is contracting for the inverse map.

\medskip

\noindent\textit{Proof of Theorem~\ref{t;main}.} Assume that $M$ is a strongly pseudoconvex almost CR manifold and $\vp$ is a contracting CR automorphism at $o\in M$. Let $\theta$ be a canonical pseudo-Hermitian structure on $\W$ as obtained in Section~\ref{subsec:canonical contact form} so that 
\begin{equation*}
\vp^*\theta = \mu \theta
\end{equation*}
on the stable manifold $\W$ for some constant $0<\mu<1$. We denote by $\|\cdot\|_\theta$ the norm of tensors measured by $\theta$. For instance, for $T = \ind{T}{\beta}{\alpha}{\gamma}\theta^\be\w\theta^\ga \otimes Z_\al$, 
$$\|T\|_\theta^2 = \ind{T}{\beta}{\alpha}{\gamma}\ind{T}{\bar\rho}{\bar\eta}{\bar\sigma} \, g_{\al\bar\eta}g^{\be\bar\rho} g^{\ga\bar\sigma}.$$
Obviously, $\|\cdot\|_\theta$ does not depend on the choice of coframe $\{\theta^\al\}$.

Let $U$ be a sufficiently small neighborhood of $o$. Since $\mu$ is constant, $\vp^{-k}(U)$ strictly increases as $k$ increases. Therefore, for a point $x\in \W$, there exists $k_0 \geq 1$ such that $x\in \vp^{-k}(U)$ for every $k\geq k_0$. Let $\theta_k = (\vp^{-k})^*\theta = \mu^{-k}\theta$. Since 
\begin{equation*}
\vp^k : (\vp^{-k}(U), \theta) \rightarrow (U, \theta_k)
\end{equation*}
is a pseudohermitain equivalence, we have 
\begin{equation*}
\|T(x)\|_\theta = \|T_k(x_k)\|_{\theta_k},
\end{equation*}
where $ x_k =\vp^k(x) \in U$ and $T_k$ is the torsion tensor for $\theta_k$ on $U$. Since $\theta_k = \mu^{-k}\theta$ a psudoconformal change of $\theta$, Proposition~\ref{p;change} yields that 
\begin{equation*}
\|T_k(x_k)\|_{\theta_k} = \mu^{k/2}\|T(x_k)\|_\theta.
\end{equation*}
Therefore, 
\begin{equation*}
\|T(x)\|_\theta = \mu^{k/2}\|T(x_k)\|_\theta \leq \mu^{k/2}\sup_{y\in U}\|T(y)\|_\theta \rightarrow 0
\end{equation*}
as $k\rightarrow\infty$. This means that $T\equiv 0$ on $\W$. Similarly, we can show that 
\begin{equation*}
\ind{N}{\bar\beta}{\alpha}{\bar\gamma} = \ind{A}{}{\alpha}{\bar\beta}
=\ind{B}{}{\alpha}{\beta}= p_{\al\be;\ga} = \ind{R}{\beta}{\alpha}{\gamma\bar\sigma}
\equiv 0
\end{equation*}
on $\W$. Therefore, we can conclude that $(M,\theta)$ is locally equivalent to $(\mh_P, \theta_P)$ from Proposition \ref{p;equiv}. We may assume $(U, \theta)$ is pseudo-Hermitian equivalent to $(V, \theta_P)$ for some neighborhood $V$ of $0$ in $\mh_P$ and let 
\begin{equation*}
F : (U,\theta)\rightarrow (V, \theta_P)
\end{equation*}
be a pseudo-Hermitian equivalence map such that $F(o) =0$. Let $\Lambda : \mh_P\rightarrow \mh_P$ be the contracting CR automorphism on $\mh_P$ defined by
\begin{equation}\label{e;lambda}
\Lambda(z,t) = (\sqrt\mu\, z, \mu t).
\end{equation}
Note that 
\begin{equation*}
\Lambda^* \theta_P = \mu \theta_P.
\end{equation*}
Therefore, the map 
\begin{equation*}
F_k := \Lambda^{-k}\circ F\circ \vp^k : \vp^{-k}(U) \rightarrow \Lambda^{-k}(V)
\end{equation*}
is a pseudo-Hermitian equivalence between $(\vp^{-k}(U),\theta)$ and $(\Lambda^{-k}(V), \theta_P)$. Since $\{F_k : k\geq 1\}$ is not compactly divergent, ($F^k(o) =0$ for all $k$) we conclude that $\{F_k\}$ has a subsequence converging to a pseudo-Hermitian equivalence $\wt F : (\W, \theta)\rightarrow (\mh_P, \theta_P)$. Altogether, we have proved Theorem~\ref{t;main}. \qed

\medskip

\begin{remark}
Unlike Theorem \ref{t;jl}, it is not very clear how to characterize the ambient manifold $M$ of arbitrary dimension. According to former results in \cite{Schoen1995} and \cite{JooLee2015}, it may be natural to expect either $M=\W\simeq\mh_P$ for some $P$ in case $M$ is noncompact, or $M$ is CR equivalent with the standard sphere and $\W = M\setminus\{\mbox{pt}\}\simeq \mh_0 $ in case $M$ is compact. In \cite{Schoen1995, JooLee2015}, this global characterization could be done from the derivative estimates obtained from a PDE theory on the CR or subconformal Yamabe equations (see Proposition 2.1 and 2.1' in \cite{Schoen1995}). On the other hand, it is still unknown whether there exists a curvature invariant which satisfies a Yamabe-type equation under pseudoconformal changes or not, if $M$ is a strongly pseudoconvex almost CR manifold of arbitrary dimension.
\end{remark}

If $M$ is the standard sphere and if $\vp$ is a contracting CR automorphism with a contracting fixed point $o$, then $\vp^{-1}$ is also a contracting CR automorphism with another contracting fixed point $o'$. The next theorem states that the converse is also true even in almost CR cases.  

\begin{thm}\label{t;thm2} Let $M$ be a strongly pseudoconvex almost CR manifold and let $\vp$ be a contracting CR automorphism with a contracting fixed point $o$. Let $\W$ be the stable manifold of $\vp$ with respect to $o$. Suppose that $\W\neq M$ and that there is a contracting fixed point $o'\in\di\W$ of $\vp^{-1}$. Then $M$ is CR equivalent to the standard sphere $S^{2n+1}\subset \C^{n+1}$.
\end{thm}

\begin{proof} Let $\widetilde F:(\W,\theta)\to (\mh_P,\theta_P)$ be a pseudo-Hermitian equivalence for the canonical contact form $\theta$ of $o$. Applying Theorem~\ref{t;main}, we have also a pseudo-Hermitian equivalence $\wt F' : (\W', \theta')\rightarrow (\mh_{P'}, \theta_{P'})$ where $\W'$ is the stable neighborhood of $o'$ with respect to $\varphi^{-1}$ and $\theta'$ is the canonical contact form on $\W'$ with 
\begin{equation*}
(\varphi^{-1})^*\theta' = \mu'\theta'
\end{equation*}
for some $0<\mu'<1$. Let $\V=\W\cap\W'$ which is an open subset of $M$ admitting $o,o'$ as  boundary points. Take a point $p\in\V$. Then 
\begin{equation}\label{orbit}
\varphi^k(p)\to o \quad\text{and}\quad \varphi^{-k}(p)\to o'
\end{equation}
as $k\to\infty$. Let us consider
\begin{enumerate}
\item $\cU=\wt F(\V)$, $\cU' = \wt F'(\V)$ open subsets of $\mh_P$ and $\mh_{P'}$, respectively
\item $(z_0,t_0) =  \wt F(p) \in\cU$, $(w_0,s_0)=\wt F'(p)\in\cU'$,
\item $\psi = \wt F\circ\varphi\circ \wt F^{-1}$, $\psi' = \wt F'\circ\varphi\circ \wt F'^{-1}$ corresponding CR automorphisms of $\varphi$ in $\mh_P$ and $\mh_{P'}$, respectively.
\end{enumerate}
Since $\cU$ and $\cU'$ are open in $\CC^n\times\RR$, we may assume that $z_0\neq 0$ and $w_0\neq 0$. Theorem~\ref{thm:CR automorphism group} implies that
\begin{equation*}
\psi(z,t) = ( \mu^{1/2}U(z), \mu t)\;, \quad \psi'(w,s) = (\mu'^{-1/2} U'(w), \mu'^{-1} s)
\end{equation*}
for some $U,U'\in\UG(n)$. Since $z_0,w_0\neq 0$, we have 
\begin{equation*}
z_k = \mu^{k/2} U^k(z_0)\to 0 \quad\text{and}\quad w_k = \mu'^{-k/2} U^{k}(w_0)\to \infty
\end{equation*}
in $\CC^n$.

Let us consider the local CR diffeomorphism $G=\wt F'\circ \wt F^{-1}:\cU\subset\mh_P\to \cU'\subset\mh_{P'}$ and denote by
\begin{equation*}
G(z,t) = (w,s)=(w^1,\ldots,w^n,s) \;.
\end{equation*}
Now suppose that $\mh_P$ is non-integrable and let $\cU_1$ be the projection image of $\cU$ to $\CC^n$. Then $w=(w^1,\ldots,w^n)$ is independent of $t$ and is a holomorphic mapping from $\cU_1$ to $\CC^n$ by (1) of Lemma~\ref{lemma:CR mapping}. Moreover
\begin{align*}
G\circ\psi^{k} &= (\wt F'\circ \wt F^{-1})\circ (\wt F\circ \varphi^k \circ \wt F^{-1})
	\\
	&= (\wt F'\circ \varphi^k \circ \wt F'^{-1})\circ (\wt F'\circ \wt F^{-1}) = \psi'^k\circ G
\end{align*}
and
\begin{align*}
G(\psi^{k}(z_0,t_0)) 
	&= \big(w(\psi^{k}(z_0,t_0)), s(\psi^{k}(z_0,t_0))\big) 
	= \big(w(z_k), s(\psi^{k}(z_0,t_0))\big) \;,
	\\
\psi'^k(G(z_0,t_0)
	&= \psi'^k(w_0,s_0) = (\mu'^{-k/2}U'^k(w_0),\mu'^{-k} s_0) \;.
\end{align*}
Therefore we have $w(z_k)=w_k=\mu'^{-k/2} U^{-k}(w_0)$. As a conclusion, the holomorphic mapping $w:\cU_1\to\CC^n$ is not defined and diverges to infinity at the origin $0$ of $\CC^n$, namely,
\begin{equation*}
z_k\to 0   \quad\text{but}\quad w(z_k) = w_k\to\infty
\end{equation*}
in $\CC^n$. This implies that
\begin{equation*}
\sum_{\alpha,\beta=1}^n \abs{\pd{w^\beta}{z^\alpha}(z_k)}^2 \to\infty \quad\text{as $k\to\infty$.}
\end{equation*}
It is a contradiction to Statement (3) of Lemma~\ref{lemma:CR mapping}. This implies $\mh_P$ is integrable. Similarly  $\mh_{P'}$ is also an integrable Heisenberg group and as a consequence, $\W\cup\W'$ is an open submanifold of $M$ whose CR structure is integrable. Notice that the automorphism $\varphi$ acts on $\W\cup\W'$ and generates a noncompact orbit with two fixed points $o$ and $o'$. Therefore $\W\cup\W'$ must be CR equivalent to the standard sphere by Schoen's theorem \cite{Schoen1995}. Then the conclusion follows since $\W\cup\W'$ is an open submanifold of $M$ without boundary and hence $M=\W\cup\W'$.
\end{proof}


\end{document}